\documentclass[oneside,english]{amsart}

\usepackage{amsthm}
\usepackage{amssymb}
\usepackage{amsmath}
\usepackage{stmaryrd}
\usepackage{xstring}
\newtheorem{thm}{Theorem}[section]
  \theoremstyle{definition}
  \newtheorem{example}[thm]{Example}
  \theoremstyle{plain}
  \newtheorem*{question*}{Question}
  \theoremstyle{definition}
  \newtheorem{defn}[thm]{Definition}
  \theoremstyle{plain}
  \newtheorem{lem}[thm]{Lemma}
  \theoremstyle{plain}
  \newtheorem{prop}[thm]{Proposition}
  \theoremstyle{plain}
  \newtheorem{cor}[thm]{Corollary}

  \theoremstyle{plain}
  \newtheorem{question}[thm]{Question}
  \theoremstyle{plain}
  \newtheorem{conjecture}[thm]{Conjecture}
  \theoremstyle{plain}
  \newtheorem{problem}[thm]{Problem}

\newtheorem*{thmaux}{Theorem \theoremauxnum}
\newtheorem*{coraux}{Corollary \theoremauxnum}
\newtheorem*{propaux}{Proposition \theoremauxnum}

\gdef\theoremauxnum{1}

\newenvironment{thmff}[2]{%
  \def\theoremauxnum{\ref{#2}}
  \begin{thmaux}[#1]
}{%
  \end{thmaux}
}

\newenvironment{corff}[2]{%
  \def\theoremauxnum{\ref{#2}}
  \begin{coraux}[#1]
}{%
  \end{coraux}
}

\newcommand{\amsdoc}{amsart}
\newcommand{\sep}{, }

\allowdisplaybreaks

\begin{document}
	\global\long\def\norm#1{\left\Vert #1\right\Vert }%
\global\long\def\abs#1{\left|#1\right|}%
\global\long\def\set#1#2{\left\{  \vphantom{#1}\vphantom{#2}#1\right.\left|\ #2\vphantom{#1}\vphantom{#2}\right\}  }%
\global\long\def\sphere#1{\mathbf{S}_{#1}}%
\global\long\def\closedball#1{\mathbf{B}_{#1}}%
\global\long\def\openball#1{\mathbb{B}_{#1}}%
\global\long\def\duality#1#2{\left\langle \vphantom{#1}\vphantom{#2}#1\right.\left|\ #2\vphantom{#1}\vphantom{#2}\right\rangle }%
\global\long\def\parenth#1{\left(#1\right)}%
\global\long\def\curly#1{\left\{  #1\right\}  }%
\global\long\def\blockbraces#1{\left[#1\right] }%
\global\long\def\span{\textup{span}}%
\global\long\def\image{\textup{Im}}%
\global\long\def\support{\textup{supp}}%
\global\long\def\N{\mathbb{N}}%
\global\long\def\R{\mathbb{R}}%
\global\long\def\Rnonneg{\mathbb{R}_{\geq0}}%
\global\long\def\C{\mathbb{C}}%
\global\long\def\tocorr{\mathbb{\ \twoheadrightarrow\ }}%
	\global\long\def\zabreiko{Zabre\u{\i}ko}%
\global\long\def\ellinfty#1#2{\ell^{\infty}(#1,#2)}%
\global\long\def\ellone#1#2{\ell^{1}(#1,#2)}%
\global\long\def\c#1#2{\mathbf{c}(#1,#2)}%
\global\long\def\czero#1#2{\mathbf{c}_{0}(#1,#2)}%
\global\long\def\ellinftyomega#1{\ellinfty{\Omega}{#1}}%
\global\long\def\elloneomega#1{\ellone{\Omega}{#1}}%
\global\long\def\comega#1{\c{\Omega}{#1}}%
\global\long\def\czeroomega#1{\czero{\Omega}{#1}}%
\global\long\def\ellinftyN#1{\ellinfty{\N}{#1}}%
\global\long\def\elloneN#1{\ellone{\N}{#1}}%
\global\long\def\cN#1{\c{\N}{#1}}%
\global\long\def\czeroN#1{\czero{\N}{#1}}%
\global\long\def\conenorm#1{\left\llbracket #1\right\rrbracket }%
\global\long\def\jamesseq{\mathcal{P}}%
\global\long\def\finite#1{\mathcal{F}\parenth{#1}}%

	\newcommand{\Miek}{Miek Messerschmidt}

\newcommand{\MiekEmail}{mmesserschmidt@gmail.com}

\newcommand{\UPAddress}{Department of Mathematics and Applied Mathematics; University of Pretoria; Private~bag~X20 Hatfield; 0028 Pretoria; South Africa}

\newcommand{\NWUAddressLineBreaks}{%
Department of Mathematics and Applied Mathematics\\
University of Pretoria\\
Private bag X20 Hatfield\\ 
0028 Pretoria\\ 
South Africa}

\newcommand{\ClaudeLeonAck}{The author's research was funded by The Claude Leon Foundation.}

	\newcommand{\paperabstract}{%
A version of the classical Klee-And\^o Theorem states the following:
For every Banach space $X$, ordered by a closed generating cone $C\subseteq X$,
there exists some $\alpha>0$ so that, for every $x\in X$, there
exist $x^{\pm}\in C$ so that $x=x^{+}-x^{-}$ and $\norm{x^{+}}+\norm{x^{-}}\leq\alpha\norm x$.

The conclusion of the Klee-And\^o Theorem is what is known as a \emph{conormality}
property. 

We prove stronger and somewhat more general versions of the Klee-And\^o
Theorem for both conormality and coadditivity (a property that is
intimately related to conormality). A corollary to our result shows
that the functions $x\mapsto x^{\pm}$, as above, may be chosen to
be bounded, continuous, and positively homogeneous, with a similar
conclusion yielded for coadditivity. Furthermore, we show that the
Klee-And\^o Theorem generalizes beyond ordered Banach spaces to Banach
spaces endowed with arbitrary collections of cones. Proofs of our
Klee-And\^o Theorems are achieved through an Open Mapping Theorem
for cone-valued multi-functions/correspondences.

We very briefly discuss a potential further strengthening of The Klee-And\^o
Theorem beyond what is proven in this paper, and motivate a conjecture
that there exists a Banach space $X$, ordered by a closed generating
cone $C\subseteq X$, for which there exist no Lipschitz functions
$(\cdot)^{\pm}:X\to C$ satisfying $x=x^{+}-x^{-}$ for all $x\in X$.
}%
\newcommand{\MSECodesPrimary}{%
	46B20\sep 
	46A30 
}%
\newcommand{\MSECodesSecondary}{%
	46B40\sep 
	32A12 
}%
\newcommand{\paperkeywords}{%
ordered Banach space\sep\ %
conormality\sep\ %
coadditivity\sep\ %
open mapping theorem} %

	\author{\Miek}
	\address{\Miek; \UPAddress}
	\email{\MiekEmail}
	\thanks{\ClaudeLeonAck}

	\title[Strong Klee-And\^o Theorems]{Strong Klee-And\^o Theorems through an Open Mapping Theorem for cone-valued multi-functions}
	\begin{abstract}
		\paperabstract
	\end{abstract}

	\subjclass[2010]{\MSECodesPrimary (primary), and \MSECodesSecondary (secondary)}
	\maketitle

\section{Introduction}

Let $X$ be a Banach space and $\curly{C_{\omega}}_{\omega\in\Omega}$
a collection of closed cones in $X$, indexed by a set $\Omega$.
Consider the following two geometric properties\textbf{\emph{}}\footnote{Historically the properties ``conormality'' and ``coadditivity'' are
respectively called ``$\alpha$-generating'' and ``$\alpha$-directedness''.
Our reason for deviating is in favor of the mnemonic device connecting
``conormality'' and ``coadditivity'' to their dual properties ``normality''
and ``additivity'' (standard terms which we do not define or need
in this paper): Roughly, a space is normal (additive) if and only
if its dual is conormal (coadditive), and vice versa. The interested
reader is referred to \cite{MesserschmidtGeometricDuality} and \cite{MesserschmidtNormality}
which tries to reference the previously existing literature on such
dualities as completely as possible.} $\curly{C_{\omega}}_{\omega\in\Omega}$ could satisfy in $X$:
\begin{enumerate}
\item \textbf{\emph{Conormality: }}There exists a constant $\alpha>0$ such
that, for every $x\in X$, there exists a decomposition $x=\sum_{\omega\in\Omega}c_{\omega}$
with $\sum_{\omega\in\Omega}\norm{c_{\omega}}\leq\alpha\norm x$ and
$c_{\omega}\in C_{\omega}$ for all $\omega\in\Omega$.\medskip
\item \textbf{\emph{Coadditivity:}} For some normed subspace $Z$ of $X^{\Omega}$,
there exists a constant $\alpha>0$ such that, for every $\xi\in Z$,
there exists some $x\in\bigcap_{\omega\in\Omega}(\xi_{\omega}+C_{\omega})$
with $\norm x\leq\alpha\norm{\xi}$.
\end{enumerate}
The following two very simple examples are easily seen to be conormal
and coadditive respectively, and also illustrates that these properties
make sense even outside the realm of classical ordered Banach spaces.
\begin{example}
\label{conormality-example}For $\R^{2}$ with the Euclidean norm
with $\{e_{1},e_{2}\}$ the standard basis for $\R^{2}$, set $\Omega:=\{e_{1},e_{2},-(e_{1}+e_{2})\}\subseteq\R^{2}$,
and define the cones $\{C_{x}\}_{x\in\Omega}$ by setting $C_{x}:=\set{\lambda x}{\lambda\geq0}$
for all $x\in\Omega$. The space $\R^{2}$ with $\{C_{x}\}_{x\in\Omega}$
is easily seen to be conormal.
\end{example}

{}
\begin{example}
\label{coadditivity-example}
	\ifdefined\amsdoc %
		For $\R^{2}$ with the Euclidean norm, we will define the cones $C_{1}:=\set{(x,y)\in\R^{2}}{x,y\geq0}$
		and $C_{2}:=\set{\alpha(1,1)+\beta(1,-1)\in\R^{2}}{\alpha,\beta\geq0}$.
		The space $\R^{2}$ with $\{C_{j}\}_{j\in\{1,2\}}$ is easily seen
		to be coadditive (with $Z$ taken as the $\ell^{2}$-direct sum two
		copies of $\R^{2}$).
	\else 
		For $\R^{2}$ with the Euclidean norm,
		we define the cones $C_{1}:=\set{(x,y)\in\R^{2}}{x,y\geq0}$ and $C_{2}:=\set{\alpha(1,1)+\beta(1,-1)\in\R^{2}}{\alpha,\beta\geq0}$.
		The space $\R^{2}$ with $\{C_{j}\}_{j\in\{1,2\}}$ is easily seen
		to be coadditive (with $Z$ taken as the $\ell^{2}$-direct sum two
		copies of $\R^{2}$).%
	\fi 
\end{example}

\medskip

The following result dates back to And\^o's \cite[Lemma~1]{Ando},
which was proven by employing an Open Mapping Theorem due to Klee
\cite[(3.2)]{Klee}. The reader should recognize the conclusion of
the following result as a conormality property.
\begin{thm}[{Klee-And\^o Theorem \cite[Lemma~1]{Ando}, \cite[(3.2)]{Klee}}]
\label{thm:classicalKleeAndo} Let $X$ be a Banach space ordered
by a closed cone $C\subseteq X$. If $C$ is generating, i.e., $X=C-C$,
then there exists a constant $\alpha>0$ so that, for every $x\in X$,
there exists a decomposition $x=x^{+}-x^{-}$ with $x^{\pm}\geq0$
and $\norm{x^{+}}+\norm{x^{-}}\leq\alpha\norm x$.
\end{thm}

Recently, a stronger version of the Klee-And\^o Theorem, stated as
Corollary~\ref{cor:klee-ando-for-conorm} below, was proven by the
author and de Jeu in \cite{deJeuMesserschmidtOpenMapping}. Its proof
employed a generalization of Banach's classical Open Mapping Theorem
\cite[Theorem~3.2]{deJeuMesserschmidtOpenMapping} together with Michael's
Selection Theorem which stated in this paper as Theorem~\ref{thm:Michael's-Selection-Theorem}.
We note that Corollary~\ref{cor:klee-ando-for-conorm} has wider
applicability than ordered Banach spaces (cf. Example~\ref{conormality-example})
and has a distinctly more geometrical flavor than the order theoretic
Theorem~\ref{thm:classicalKleeAndo} (where we essentially restrict
our attention to only two cones: $C$ and $-C$).

\begin{corff}{Strong Klee-And\^o Theorem for conormality {\cite[Theorem~4.1]{deJeuMesserschmidtOpenMapping}}}{cor:klee-ando-for-conorm}

Let $X$ be a Banach space and $\curly{C_{\omega}}_{\omega\in\Omega}$
an indexed collection of closed cones in $X$. The following are equivalent:
\begin{enumerate}
\item \label{decomp}For every $x\in X$, there exists a decomposition $x=\sum_{\omega\in\Omega}c_{\omega}$,
with $c_{\omega}\in C_{\omega}$ for every $\omega\in\Omega$, and
satisfying $\sum_{\omega\in\Omega}\norm{c_{\omega}}<\infty$.\medskip
\item \label{conormality--in--theorem}There exists an $\alpha>0$ such
that, for every $x\in X$, there exists a decomposition $x=\sum_{\omega\in\Omega}c_{\omega}$,
with $c_{\omega}\in C_{\omega}$ for every $\omega\in\Omega$, and
satisfying $\sum_{\omega\in\Omega}\norm{c_{\omega}}\leq\alpha\norm x$.\medskip
\item There exists an $\alpha>0$ and, for every $\omega\in\Omega$, there
exists a continuous positively homogeneous map $\delta_{\omega}:X\to C_{\omega}$
such that, for every $x\in X$, we have $x=\sum_{\omega\in\Omega}\delta_{\omega}(x)$
and $\sum_{\omega\in\Omega}\norm{\delta_{\omega}(x)}\leq\alpha\norm x$.
\end{enumerate}
\end{corff}

In other words, the mere fact that one can decompose arbitrary elements
of $X$ as the limit of absolutely convergent series with terms chosen
from the closed cones $\{C_{\omega}\}_{\omega\in\Omega}$, automatically
implies that one can always choose such a decomposition in a bounded,
continuous and positively homogeneous and manner. This is, of course,
particularly useful when considering spaces of continuous functions
taking values in an ordered Banach space (cf. \cite[Corollary~2.8]{CPIII}).

Having now discussed the Strong Klee-And\^o Theorem for conormality,
we point out that coadditivity and conormality are intimately related,
in that both their statements satisfy the following general template: 
\begin{quote}
For every structure $A$ of a certain type, there exists some related
structure $B$, and this $B$ is bounded, in some sense, by $A$.
\end{quote}
This relationship and the fact that the Klee-And\^o Theorem for conormality
was previously established in \cite{deJeuMesserschmidtOpenMapping},
raises the following question:
\begin{question}
Does there there exist a Klee-And\^o type theorem for coadditivity?
Roughly, if the intersection of certain translates of a collection
of closed cones in a Banach space is always non-empty, can one always
find an element in such an intersection in a bounded, continuous and
positively homogeneous manner?
\end{question}

We will answer this question positively in this paper by proving Corollary~\ref{cor:klee-ando-for-coadd}.
\begin{corff}{Strong Klee-And\^o Theorem for coadditivity}{cor:klee-ando-for-coadd}

\label{thm:coadditivity} Let $X$ be a Banach space and $\{C_{\omega}\}_{\omega\in\Omega}$
an indexed collection of closed cones in $X$. Let $Z$ be either
of the spaces $\comega X$ or $\ellinftyomega X$. Then the following
are equivalent:
\begin{enumerate}
\item For every $\xi\in Z$, the intersection $\bigcap_{\omega\in\Omega}(\xi_{\omega}+C_{\omega})$
is non-empty.\medskip
\item \label{coadditivity-in-theorem}There exists an $\alpha>0$ such that,
for every $\xi\in Z$, there exists some ${y\in\bigcap_{\omega\in\Omega}(\xi_{\omega}+C_{\omega})}$
with $\norm y\leq\alpha\norm{\xi}_{\infty}$.\medskip
\item There exists an $\alpha>0$ and a continuous positively homogeneous
map $\upsilon:Z\to X$ such that, for every $\xi\in Z$, we have $\upsilon(\xi)\in\bigcap_{\omega\in\Omega}(\xi_{\omega}+C_{\omega})$
and $\norm{\upsilon(\xi)}\leq\alpha\norm{\xi}_{\infty}$.
\end{enumerate}
\end{corff}

Our approach in proving Corollary~\ref{cor:klee-ando-for-coadd}
is similar in broad strokes to the proof of \cite[Theorem~4.1]{deJeuMesserschmidtOpenMapping},
stated in this paper as Corollary~\ref{cor:klee-ando-for-conorm}.
However, the Open Mapping Theorem \cite[Theorem 3.2]{deJeuMesserschmidtOpenMapping}
employed in the proof of \cite[Theorem~4.1]{deJeuMesserschmidtOpenMapping}
is not strong enough to establish the results leading up to Corollary~\ref{cor:klee-ando-for-coadd}.
We further strengthen \cite[Theorem~3.2]{deJeuMesserschmidtOpenMapping}
to the version stated as Theorem~\ref{thm:openmappingtheorem}, which
we use to establish both our Strong Klee-And\^o Theorems for both
coadditivity and conormality:

\begin{thmff}{Open Mapping Theorem for cone-valued correspondences}{thm:openmappingtheorem}

Let $C$ a complete metric cone (as defined in Section~\ref{sec:metric-cones-and-open-mapping-thm}),
$Y$ a Banach space and $D\subseteq Y$ a closed cone. Let $T:C\to Y$
be a continuous additive positively homogeneous map. If the correspondence
$\Psi:C\tocorr Y$ defined by $\Psi(c):=Tc+D\ (x\in C)$ is surjective
(in the sense that, for every $y\in Y$, there exists some $c\in C$
such that $y\in\Psi(c)$), then the image under $\Psi$ of the open
unit ball about zero in $C$ is open in $Y$.

\end{thmff}

\medskip

We point out the following problem toward further strengthening of
the Klee-And\^o Theorems proven in this paper. This problem is motivated
by the observation that the lattice operations in every Banach lattice
are Lipschitz \cite[Proposition II.5.2]{Schaefer}, and the further
question as to whether a form of this carries over to ordered Banach
spaces, or more generally, to Banach spaces endowed with arbitrary
collections of cones. To the author's knowledge, this problem remains
open:
\begin{problem}
\label{que:works_for_lip?}Do Corollaries \ref{cor:klee-ando-for-coadd}
and \ref{cor:klee-ando-for-conorm} remain true when, in their statements,
the word ``continuous'' is replaced by the word ``Lipschitz''?
\end{problem}

A straightforward application of \cite[Corollary~4.6]{MesserschmidtPtwiseLipSelpub}
shows that Corollaries \ref{cor:klee-ando-for-coadd} and~\ref{cor:klee-ando-for-conorm}
remain true when, in their statements, the word ``continuous'' is
replaced with the phrase ``continuous and pointwise Lipschitz on a
dense set of its domain''. However, an example devised by Aharoni
and Lindenstrauss (cf.~\cite{LindenstraussAharoni} and~\cite[Example~1.20]{LindenstraussBenyamini})
shows that \cite[Corollary~4.6]{MesserschmidtPtwiseLipSelpub} cannot
be improved, in general, to a result yielding functions that are Lipschitz.
We refer the reader to \cite{MesserschmidtPtwiseLipSelpub} for more
detail. 

The above-mentioned example by Aharoni and Lindenstrauss and the difficulties
presented in attempts at solving Problem~\ref{que:works_for_lip?}
positively, motivates the following conjecture which, if true, is
sufficient to solve Problem~\ref{que:works_for_lip?} negatively:
\begin{conjecture}
There exists a Banach space $X$, ordered by a closed generating cone
$C\subseteq X$, for which there exist no Lipschitz functions $(\cdot)^{\pm}:X\to C$
satisfying $x=x^{+}-x^{-}$ for all $x\in X$.
\end{conjecture}

\medskip

We briefly describe the structure of the paper.

\medskip

In Section~\ref{sec:preliminaries} we provide some preliminary definitions
and notation used throughout the current paper.

Sections~\ref{sec:bounded-and-additive-correspondences} and~\ref{sec:metric-cones-and-open-mapping-thm}
sees the introduction of some terminology on correspondences (also
known as multi-functions) and will prove some general results. Section~\ref{sec:bounded-and-additive-correspondences}
will introduce what we call ``additive'' and ``$\alpha$-bounded''
correspondences and will establish some general results that we will
use in later sections. Section~\ref{sec:metric-cones-and-open-mapping-thm}
sees the definition of metric cones and proof of one of our main results,
an Open Mapping Theorem for cone-valued correspondences (Theorem~\ref{thm:openmappingtheorem}).

We apply our results from the previous sections to prove our Strong
Klee-And\^o Theorems for conormality and coadditivity in Section~\ref{sec:strong-klee-ando-theorem}.
Although our Strong Klee-And\^o Theorem for conormality was previously
established in \cite{deJeuMesserschmidtOpenMapping}, we also include
a proof of it here for the sake of completeness.

\section{Preliminary definitions and notation\label{sec:preliminaries}}

All vector spaces are assumed to be over the reals. 

Let $V$ be a vector space. A non-empty subset $C\subseteq V$ satisfying
$C+C\subseteq C$ and $\lambda C\subseteq C$ for all $\lambda\geq0$
will be called a \emph{cone}. If $C\cap-C=\{0\}$, then we will say
$C$ is a \emph{proper cone. }If $V=C-C$, then we will say that $C$
is \emph{generating} in $V$.\emph{ }Translation invariant and positively
homogeneous pre-orders (partial orders) on $V$ are easily seen to
be in bijection with cones (proper cones) in $V$, cf. \cite[Section~1.1]{AliprantisTourky}.
We will say $V$ \emph{is ordered by a cone }$C$ by defining ``$v\leq w$''
to mean $w\in v+C$ for $v,w\in V$. 

Let $\phi:V\to\R$ be any map. The map $\phi$ is said to be \emph{positively
homogeneous,} if $\phi(\lambda v)=\lambda\phi(v)$ for all $v\in V$
and $\lambda\geq0$. The map $\phi$ is said to be\emph{ subadditive}
if $\phi(v+w)\leq\phi(v)+\phi(w)$ for all $v,w\in V$. 

For a normed space $X$, we will denote the unit sphere, closed unit
ball and open unit ball of $X$ respectively by $\sphere X$, $\closedball X$
and $\openball X$.

Let $\Omega$ be an arbitrary index set, and let $X$ be a Banach
space. 
\begin{enumerate}
\item For $1\leq p\leq\infty$, by $\ell^{p}(\Omega,X)$ we will denote
the usual $\ell^{p}$-direct sum of $\abs{\Omega}$ copies of $X$,
with the norm on $\ell^{p}(\Omega,X)$ denoted by $\norm{\cdot}_{p}$. 
\item By $\comega X$ we will denote the closed subspace of $\ellinftyomega X$
of all elements $\xi\in\ellinftyomega X$ for which there exists some
$x\in X$ such that, for every $\varepsilon>0$, the set $\set{\omega\in\Omega}{\norm{\xi_{\omega}-x}\geq\varepsilon}$
is finite. 
\end{enumerate}
Let $\curly{C_{\omega}}_{\omega\in\Omega}$ be an indexed collection
of cones in $X$, and let $Z$ be some vector subspace of $X^{\Omega}$.
The notation ``$\bigoplus_{\omega\in\Omega}C_{\omega}\subseteq Z$''
will be used to denote the set $\set{\xi\in Z}{\forall\omega\in\Omega,~\xi_{\omega}\in C_{\omega}}$.

\section{Bounded and additive correspondences\label{sec:bounded-and-additive-correspondences}}

In this section our goal is to prove the general result, Proposition~\ref{prop:additive-bounded-is-lower-hemicont},
which establishes the lower hemicontinuity of certain correspondences
constructed from given correspondences having some extra algebraic
structure.\medskip

We first introduce some terminology: Let $A,B$ be sets. By a \emph{correspondence}
we mean a set valued map $\varphi:A\to2^{B}$ and will use the notation
$\varphi:A\tocorr B$. If $A$ and $B$ are topological spaces, we
will say that a correspondence $\varphi:A\tocorr B$ is \emph{lower
hemicontinuous} if, for every $a\in A$ and every open set $U\subseteq B$
satisfying $\varphi(a)\cap U\neq\emptyset$, there exists some open
set $V\ni a$ satisfying $\varphi(v)\cap U\neq\emptyset$ for all
$v\in V$. By a \emph{continuous selection of $\varphi$ }we mean
a continuous function $f:A\to B$ with $f(a)\in\varphi(a)$ for all
$a\in A.$ 
\begin{defn}
Let $X$ and $Y$ be normed spaces and let $\varphi:X\tocorr Y$ be
a correspondence.

\begin{enumerate}
\item We will say $\varphi$ is \emph{additive} if, for $x,z\in X$, $\varphi(x)+\varphi(z)\subseteq\varphi(x+z)$.\medskip
\item For some $\alpha>0$, we will say $\varphi$ is $\alpha$\emph{-bounded}
if, for every $x\in X$ and $\varepsilon>0$, $\varphi(x)\cap(\alpha+\varepsilon)\norm x\closedball Y$
is non-empty. 
\end{enumerate}
\end{defn}

We begin with the following lemma, which is a crucial ingredient in
the proof of Proposition~\ref{prop:additive-bounded-is-lower-hemicont}.
\begin{lem}
\label{lem:convex-intersection-withball-nonempty}Let $X$ be a normed
space. Let $\alpha>0$ and $G\subseteq X$ be a convex set such that,
for every $\varepsilon>0$, the set $G\cap(\alpha+\varepsilon)\closedball X$
is non-empty. If, for an open set $U\subseteq X$ and some $\varepsilon_{0}>0$,
the set $G\cap(\alpha+\varepsilon_{0})\closedball X\cap U$ is non-empty,
then $G\cap(\alpha+\varepsilon_{0})\openball X\cap U$ is also non-empty.
\end{lem}

\begin{proof}
Let $U\subseteq X$ be open and $\varepsilon_{0}>0$ such that $G\cap(\alpha+\varepsilon_{0})\closedball X\cap U\neq\emptyset$.
Let $x\in G\cap(\alpha+\varepsilon_{0})\closedball X\cap U$. If $x\in(\alpha+\varepsilon_{0})\openball X$,
then we are done. We therefore assume that $x\in(\alpha+\varepsilon_{0})\sphere X$.
Let $y\in G\cap(\alpha+2^{-1}\varepsilon_{0})\closedball X\neq\emptyset$.
Then, for every $t\in(0,1]$, 
\begin{eqnarray*}
\norm{ty+(1-t)x} & \leq & t\norm y+(1-t)\norm x\\
 & \leq & t(\alpha+2^{-1}\varepsilon_{0})+(1-t)(\alpha+\varepsilon_{0})\\
 & < & t(\alpha+\varepsilon_{0})+(1-t)(\alpha+\varepsilon_{0})\\
 & = & (\alpha+\varepsilon_{0}).
\end{eqnarray*}
In other words, $ty+(1-t)x\in(\alpha+\varepsilon_{0})\openball X$
for all $t\in(0,1]$. Since $[0,1]\ni t\mapsto ty+(1-t)x$ is continuous,
there exists some $t_{0}\in(0,1]$ such that $t_{0}y+(1-t_{0})x\in U$.
Since $G$ is convex, $t_{0}y+(1-t_{0})x\in G$. We conclude $t_{0}y+(1-t_{0})x\in G\cap(\alpha+\varepsilon_{0})\openball X\cap U$.
\end{proof}
\begin{prop}
\label{prop:additive-bounded-is-lower-hemicont}Let $X$ and $Y$
be normed spaces and $\alpha>0$. Let $\varphi:X\tocorr Y$ be a convex-valued
additive $\alpha$-bounded correspondence. Then, for every $\varepsilon>0$,
the correspondence $\varphi_{\varepsilon}:\sphere X\tocorr Y$, defined
by 
\[
\varphi_{\varepsilon}(x):=\varphi(x)\cap(\alpha+\varepsilon)\closedball Y\quad(x\in\sphere X),
\]
is non-empty\textendash{} and convex-valued and lower hemicontinuous.
\end{prop}

\begin{proof}
Since $\varphi$ is $\alpha$-bounded and convex-valued, that $\varphi_{\varepsilon}$
is non-empty\textendash{} and convex-valued for every $\varepsilon>0$
is immediate.

We establish the lower hemicontinuity of $\varphi_{\varepsilon}$.
Let $\varepsilon>0$ and $x\in\sphere X$ be arbitrary. Let $U\subseteq Y$
be open such that $\varphi_{\varepsilon}(x)\cap U=\varphi(x)\cap(\alpha+\varepsilon)\closedball Y\cap U$
is non-empty. By Lemma~\ref{lem:convex-intersection-withball-nonempty},
$\varphi(x)\cap(\alpha+\varepsilon)\openball Y\cap U$ is non-empty.
Let $y\in\varphi(x)\cap(\alpha+\varepsilon)\openball Y\cap U$ be
arbitrary, and let $r>0$ be such that $y+r\openball Y\subseteq(\alpha+\varepsilon)\openball Y\cap U$.
Now, let $z\in\sphere X$ be such that $\norm{z-x}<r(\alpha+\varepsilon)^{-1}$.
Since $\varphi$ is $\alpha$-bounded, there exists some $v\in\varphi(z-x)\cap(\alpha+\varepsilon)\norm{z-x}\closedball Y$.
Then, $\norm v\leq(\alpha+\varepsilon)\norm{z-x}<r(\alpha+\varepsilon)^{-1}(\alpha+\varepsilon)=r,$
so that $y+v\in y+r\openball Y\subseteq(\alpha+\varepsilon)\openball Y\cap U$,
and, since $\varphi$ is additive, we have $y+v\in\varphi(x+z-x)=\varphi(z)$.
Therefore $y+v\in\varphi(z)\cap(\alpha+\varepsilon)\openball Y\cap U=\varphi_{\varepsilon}(z)\cap U$.
Since $z$ was chosen arbitrarily from $V:=(x+r(\alpha+\varepsilon)^{-1}\openball X)\cap\sphere X$,
we conclude that $\varphi_{\varepsilon}$ is lower hemicontinuous.
\end{proof}

\section{An Open Mapping Theorem for cone-valued correspondences\label{sec:metric-cones-and-open-mapping-thm}}

In this section we will prove one of our main results, namely an Open
Mapping Theorem for cone-valued correspondences.

We begin with the following definitions and notation:
\begin{defn}
Let $C$ be a set equipped with operations $+:C\times C\to C$ and
$\cdot:\Rnonneg\times C\to C$ . The set $C$ will be called an \emph{abstract
cone}, if there exists an element $0\in C$ such that, for all $u,v,w\in C$
and $\lambda,\mu\in\Rnonneg$, the following hold:

\begin{enumerate}
\item $u+0=u$
\item $(u+v)+w=u+(v+w)$
\item $u+v=v+u$
\item $u+v=u+w$ implies $v=w$
\item $1u=u$
\item $(\lambda\mu)u=\lambda(\mu u)$
\item $(\lambda+\mu)u=\lambda u+\mu u$
\item $\lambda(u+v)=\lambda u+\lambda v$.
\end{enumerate}
\end{defn}

{}
\begin{defn}
Let $C$ be an abstract cone and $d$ a metric on $C$. The pair $(C,d)$
will be called a \emph{metric cone} if, for all $u,v,w\in C$ and
$\lambda\in\Rnonneg$, 
\begin{eqnarray*}
d(0,\lambda u) & = & \lambda d(0,u),\\
d(u+v,u+w) & \leq & d(v,w).
\end{eqnarray*}
We introduce the notation $\conenorm x:=d(0,x)$ for $x\in C$ and
by $\openball C$ we will denote the open unit ball about $0\in C$,
i.e., $\openball C:=\set{c\in C}{\conenorm c<1}$.
\end{defn}

Similarly to Banach spaces, if a metric cone is complete, then absolutely
convergent series always converge.
\begin{lem}
\label{lem:cone-abs-conv}Let $(C,d)$ be a complete metric cone.
If a sequence $\{c_{i}\}\subseteq C$ is such that $\sum_{i=1}^{\infty}\conenorm{c_{i}}$
converges, then the series $\sum_{i=1}^{\infty}c_{i}$ converges in
$C$.
\end{lem}

\begin{proof}
Let $C$ be a complete metric cone. Let $\{c_{i}\}\subseteq C$ be
such that $\sum_{i=1}^{\infty}\conenorm{c_{i}}<\infty.$ From the
definition of a metric cone, it is easily seen that $\curly{\sum_{i=1}^{n}c_{i}}_{n\in\N}$
is a Cauchy sequence, and hence converges in $C$. 
\end{proof}
To establish our Open Mapping Theorem, we will employ the Baire Category
Theorem in the form of \zabreiko's Lemma:
\begin{lem}[{\zabreiko's Lemma, \cite[Lemma~1.6.3]{Megginson}}]
\label{lem:Zabreiko} Every countably subadditive seminorm on a Banach
space is continuous.
\end{lem}

Finally, we will prove our main result of this section:
\begin{thm}[Open Mapping Theorem for cone-valued correspondences]
\label{thm:openmappingtheorem} Let $C$ be a complete metric cone,
$Y$ a Banach space and $D\subseteq Y$ a closed cone. Let ${T:C\to Y}$
be a continuous additive positively homogeneous map. If the correspondence
$\Psi:C\tocorr Y$ defined by $\Psi(c):=Tc+D\ (c\in C)$ is surjective,
(in the sense that, for every $y\in Y$, there exists some $c\in C$
such that $y\in\Psi(c)$), then $\Psi(\openball C)\subseteq Y$ is
open.
\end{thm}

\begin{proof}
Let $\Psi$, as defined, be surjective. Since $T$ is additive and
$D$ is a cone, it is clear that $\Psi$ is an additive correspondence.
We define the map $\rho:Y\to\Rnonneg$ by 
\[
\rho(y):=\inf\set{\conenorm c}{y\in\Psi(c)},\quad(y\in Y).
\]
We note that $\rho$ is positively homogeneous, and, by additivity
of $\Psi$, we see that $\rho$ is subadditive. Furthermore, the map
$q:Y\to\Rnonneg$, defined by $q(y):=\rho(y)\vee\rho(-y)$ for $y\in Y$,
is a seminorm on $Y$. 

We claim that that $q$ is countably subadditive. 

Let $\{y_{n}\}_{n\in\N}\subseteq Y$ be such that the series $\sum_{n=1}^{\infty}y_{n}$
converges. If $\sum_{n=1}^{\infty}q(y_{n})=\infty$ then $q\parenth{\sum_{n=1}^{\infty}y_{n}}\leq\sum_{n=1}^{\infty}q(y_{n})$
holds trivially. We may therefore assume that $\sum_{n=1}^{\infty}q(y_{n})<\infty$.
Let $\kappa\in\{\pm1\}$ be such that $q\parenth{\sum_{n=1}^{\infty}y_{n}}=\rho\parenth{\kappa\sum_{n=1}^{\infty}y_{n}}$.
Let $\varepsilon>0$ be arbitrary and, for each $n\in\N$, let $c_{n}\in C$
be such that $\kappa y_{n}\in\Psi(c_{n})$ and $\conenorm{c_{n}}<\rho(\kappa y_{n})+2^{-n}\varepsilon\leq q(y_{n})+2^{-n}\varepsilon$.
Then $\sum_{n=1}^{\infty}\conenorm{c_{n}}<\sum_{n=1}^{\infty}q(y_{n})+\varepsilon$.
Hence, by Lemma~\ref{lem:cone-abs-conv}, the series $\sum_{n=1}^{\infty}c_{n}$
converges and $\conenorm{\sum_{n=1}^{\infty}c_{n}}<\sum_{n=1}^{\infty}q(y_{n})+\varepsilon$.

For every $n\in\N$, we have $\kappa y_{n}\in\Psi(c_{n})=Tc_{n}+D$,
i.e., there exists some $d_{n}\in D$ such that $\kappa y_{n}-Tc_{n}=d_{n}$.
Because $T$ is continuous and additive, the series $\sum_{n=1}^{\infty}\parenth{\kappa y_{n}-Tc_{n}}$
converges to $\sum_{n=1}^{\infty}\kappa y_{n}-T\parenth{\sum_{n=1}^{\infty}c_{n}}$.
But $\sum_{n=1}^{\infty}d_{n}=\sum_{n=1}^{\infty}\kappa y_{n}-T\parenth{\sum_{n=1}^{\infty}c_{n}}$,
and, in particular we note that the series $\sum_{n=1}^{\infty}d_{n}$
converges. Since $D$ is a closed cone, the series $\sum_{n=1}^{\infty}d_{n}$
converges to a point in $D$. Therefore 
\[
\kappa\sum_{n=1}^{\infty}y_{n}=T\parenth{\sum_{n=1}^{\infty}c_{n}}+\sum_{n=1}^{\infty}d_{n}\in T\parenth{\sum_{n=1}^{\infty}c_{n}}+D,
\]
 and 
\[
q\parenth{\sum_{n=1}^{\infty}y_{n}}=\rho\parenth{\kappa\sum_{n=1}^{\infty}y_{n}}\leq\conenorm{\sum_{n=1}^{\infty}c_{n}}\leq\sum_{n=1}^{\infty}\conenorm{c_{n}}<\sum_{n=1}^{\infty}q(y_{n})+\varepsilon.
\]
Since $\varepsilon>0$ was chosen arbitrarily, the claim that $q$
is countably subadditive follows. We conclude that $q$ is continuous
by \zabreiko's Lemma (Lemma~\ref{lem:Zabreiko}).

By subadditivity of $\rho$, for all $x,y\in Y$ we have $|\rho(x)-\rho(y)|\leq\max\{\rho(\pm(x-y))\}=q(x-y)$,
which implies that $\rho$ is also continuous, and finally, that $\rho^{-1}([0,1))=\Psi(\openball C)$
is open.
\end{proof}

\section{Strong Klee-And\^o Theorems for coadditivity and conormality\label{sec:strong-klee-ando-theorem}}

We are now ready to establish our Strong Klee-And\^o Theorems for
conormality and coadditivity through an application of Theorem~\ref{thm:openmappingtheorem}.
Although our focus in this paper is on establishing a Klee-And\^o
Theorem for coadditivity (Corollary~\ref{cor:klee-ando-for-coadd}),
for the sake of completeness and illustration we include a proof of
a Klee-And\^o Theorem conormality, Corollary~\ref{cor:klee-ando-for-conorm}
(also proven in \cite{deJeuMesserschmidtOpenMapping}).
\begin{thm}
\label{thm:gen-klee-ando}Let $X$ be a Banach space and $\{C_{\omega}\}_{\omega\in\Omega}$
be an indexed collection of closed cones in $X$. \medskip 

\begin{enumerate}
\item Let $Z$ be either of the spaces $\comega X$ or $\ellinftyomega X$,
and let the correspondence $\Upsilon:Z\tocorr X$ be defined by 
\[
\Upsilon(\xi):=\bigcap_{\omega\in\Omega}\parenth{\xi_{\omega}+C_{\omega}}\quad(\xi\in Z).
\]
If $\Upsilon$ is non-empty-valued, then there exists some $\alpha>0$
for which $\Upsilon$ is $\alpha$-bounded. \medskip
\item Let the correspondence $\Delta:X\tocorr\elloneomega X$ be defined
by 
\[
\Delta(x):=\set{\xi\in\bigoplus_{\omega\in\Omega}C_{\omega}\subseteq\elloneomega X}{\sum_{\omega\in\Omega}\xi_{\omega}=x}\quad(x\in X).
\]
If $\Delta$ is non-empty-valued, then there exists some $\alpha>0$
for which $\Delta$ is $\alpha$-bounded. 
\end{enumerate}
\end{thm}

\begin{proof}
We prove (1) for the case that $Z=\ellinftyomega X$. The case where
$Z=\comega X$ follows similarly. 

Let $\{C_{\omega}\}_{\omega\in\Omega}$ be such that, for every $\xi\in\ellinftyomega X$,
we have $\bigcap_{\omega\in\Omega}\parenth{\xi_{\omega}+C_{\omega}}\neq\emptyset$.
Define $D:=\bigoplus_{\omega\in\Omega}(-C_{\omega})\subseteq\ellinftyomega X$
and $T:X\to\ellinftyomega X$ as $Tx:=(\omega\mapsto x)$ for all
$x\in X$ and $\omega\in\Omega$. The cone $D$ is closed in $\ellinftyomega X$,
and, since $T$ is a linear isometry, it is clear that $T$ is continuous,
additive and positively homogeneous. We define $\Psi:X\tocorr\ellinftyomega X$
by $\Psi(x):=Tx+D$ for all $x\in X$. It is easily seen that $\Psi$
is surjective: Let $\xi\in\ellinftyomega X$ be arbitrary and choose
$x\in\bigcap_{\omega\in\Omega}\parenth{\xi_{\omega}+C_{\omega}}\neq\emptyset$,
then $\xi\in Tx+D$. Now, by our Open Mapping Theorem (Theorem~\ref{thm:openmappingtheorem}),
$\Psi(\openball X)$ is open in $\ellinftyomega X$. Let $\beta>0$
be such that $\beta\openball{\ellinftyomega X}\subseteq\Psi(\openball X).$
Then, for every $\xi\in\ellinftyomega X$ and $\varepsilon>0$, there
exists some $w\in\openball X$ such that $\beta\xi/(1+\varepsilon\beta)\norm{\xi}_{\infty}\in\Psi(w)=Tw+D$.
Setting $x:=(\beta^{-1}+\varepsilon)\norm{\xi}_{\infty}w$ we obtain
$\xi\in\Psi(x)$, implying $x\in\bigcap_{\omega\in\Omega}\parenth{\xi_{\omega}+C_{\omega}}$,
and $\norm x\leq(\beta^{-1}+\varepsilon)\norm{\xi}_{\infty}.$ I.e.,
setting $\alpha:=\beta^{-1}$, we have, for any $\xi\in\ellinftyomega X$
and $\varepsilon>0$, that $\Upsilon(\xi)\cap(\alpha+\varepsilon)\norm{\xi}_{\infty}\closedball X\neq\emptyset$.
We conclude that the correspondence $\Upsilon$ is $\alpha$-bounded.

We prove (2). Let $\{C_{\omega}\}_{\omega\in\Omega}$ be such that
for every $x\in X$, there exists some $\xi\in\bigoplus_{\omega\in\Omega}C_{\omega}\subseteq\elloneomega X$
such that $x=\sum_{\omega\in\Omega}\xi_{\omega}$. We define $C:=\bigoplus_{\omega\in\Omega}C_{\omega}\subseteq\elloneomega X$
and $\Sigma:C\to X$ by $\Sigma\xi:=\sum_{\omega\in\Omega}\xi_{\omega}$
for $\xi\in\elloneomega X$. Let $\Psi:\elloneomega X\tocorr X$ be
defined by $\Psi(\xi):=\Sigma\xi+\{0\}$. The cone $C$ is a complete
metric cone with the metric induced by the $\ell^{1}$-norm, and furthermore,
the map $\Sigma$ is surjective, continuous, additive and positively
homogeneous. Therefore, by our Open Mapping Theorem (Theorem~\ref{thm:openmappingtheorem}),
$\Psi(\openball C)$ is an open set. Let $\beta>0$ be such that $\beta\openball X\subseteq\Psi(\openball C).$
Then, for every $x\in X$ and $\varepsilon>0$, there exists some
$\eta\in\openball C$ such that $\beta x/(1+\varepsilon\beta)\norm x\in\Psi(\eta)=\sum_{\omega\in\Omega}\eta_{\omega}+\{0\}$.
Setting $\xi:=(\beta^{-1}+\varepsilon)\norm x\eta$, we obtain $x\in\Psi(\xi)$,
implying $x=\sum_{\omega\in\Omega}\xi_{\omega}$, and $\norm{\xi}_{1}\leq(\beta^{-1}+\varepsilon)\norm x.$
I.e., setting $\alpha:=\beta^{-1}$, for any $x\in X$ and $\varepsilon>0$,
we obtain $\Delta(x)\cap(\alpha+\varepsilon)\norm x\closedball{\elloneomega X}\neq\emptyset$.
We conclude that the correspondence $\Delta$ is $\alpha$-bounded.
\end{proof}
We now apply Proposition~\ref{prop:additive-bounded-is-lower-hemicont},
to show that certain correspondences related to $\Upsilon$ and $\Delta$
are lower hemicontinuous.
\begin{cor}
\label{cor:upsilon-and-delta-arelower-hemo-cont}Let $X$ be a Banach
space and $\{C_{\omega}\}_{\omega\in\Omega}$ an indexed collection
of closed cones. 

\begin{enumerate}
\item Let $Z$ be either of the spaces $\comega X$ or $\ellinftyomega X$.
If $\Upsilon:Z\tocorr X$, as defined in Theorem~\ref{thm:gen-klee-ando},
is non-empty-valued, then there exists a constant $\alpha>0$ such
that, for every $\varepsilon>0$, the correspondence $\Upsilon_{\varepsilon}:\sphere Z\tocorr X$,
defined by 
\[
\Upsilon_{\varepsilon}(\xi):=\Upsilon(\xi)\cap(\alpha+\varepsilon)\closedball X\quad(\xi\in\sphere Z),
\]
is non-empty\textendash{} closed\textendash{} and convex-valued, and
is lower hemicontinuous.\medskip
\item If $\Delta:X\tocorr\elloneomega X$, as defined in Theorem~\ref{thm:gen-klee-ando},
is non-empty-valued, then there exists a constant $\alpha>0$ such
that, for every $\varepsilon>0$, the correspondence $\Delta_{\varepsilon}:\sphere X\tocorr\elloneomega X$,
defined by 
\[
\Delta_{\varepsilon}(x):=\Delta(x)\cap(\alpha+\varepsilon)\closedball{\elloneomega X}\quad(x\in\sphere X),
\]
is non-empty\textendash{} closed\textendash{} and convex-valued, and
is lower hemicontinuous.
\end{enumerate}
\end{cor}

\begin{proof}
We prove (1). By Theorem~\ref{thm:gen-klee-ando}(1), there exists
some $\alpha>0$ for which $\Upsilon$ is $\alpha$-bounded. It is
then clear that $\Upsilon_{\varepsilon}$ is then non-empty\textendash ,
closed\textendash{} and convex-valued for every $\varepsilon>0$.
It is easily seen that $\Upsilon$ is additive, so that, by Proposition~\ref{prop:additive-bounded-is-lower-hemicont},
$\Upsilon_{\varepsilon}$ is lower hemicontinuous for every $\varepsilon>0$.

We prove (2). By Theorem~\ref{thm:gen-klee-ando}(2), there exists
some $\alpha>0$ for which $\Delta$ is $\alpha$-bounded. Again,
it is clear that $\Delta_{\varepsilon}$ is non-empty\textendash ,
closed\textendash{} and convex-valued for every $\varepsilon>0$.
That $\Delta$ is additive is easily seen, so that, by Proposition~\ref{prop:additive-bounded-is-lower-hemicont},
$\Delta_{\varepsilon}$ is lower hemicontinuous for every $\varepsilon>0$.
\end{proof}
We will now apply Michael's Selection Theorem to obtain continuous
selections of $\Delta$ and $\Upsilon$.
\begin{thm}[{Michael's Selection Theorem \cite[Theorem 17.66]{AliprantisBorder}}]
\label{thm:Michael's-Selection-Theorem} A lower hemicontinuous correspondence
from a paracompact space into a Banach space with non-empty, closed
and convex values admits a continuous selection.
\end{thm}

\begin{cor}
\label{cor:raw-klee-ando}Let $X$ be a Banach space and $\{C_{\omega}\}_{\omega\in\Omega}$
an indexed collection of closed cones. 

\begin{enumerate}
\item Let $Z$ be either of the spaces $\comega X$ or $\ellinftyomega X$.
If $\Upsilon:Z\tocorr X$, as defined in Theorem~\ref{thm:gen-klee-ando},
is non-empty-valued, then there exists a constant $\alpha>0$ such
that, for every $\varepsilon>0$, there exists a continuous positively
homogeneous selection $\upsilon:Z\to X$ of $\Upsilon$ such that
$\norm{\upsilon(\xi)}\leq(\alpha+\varepsilon)\norm{\xi}_{\infty}$
for every $\xi\in Z$.\medskip
\item If $\Delta:X\tocorr\elloneomega X$, as defined in Theorem~\ref{thm:gen-klee-ando},
is non-empty-valued, then there exists a constant $\alpha>0$ such
that, for every $\varepsilon>0$, there exists a continuous positively
homogeneous selection $\delta:X\to\elloneomega X$ of $\Delta$ such
that $\norm{\delta(x)}_{1}\leq(\alpha+\varepsilon)\norm x$ for every
$x\in X$. 
\end{enumerate}
\end{cor}

\begin{proof}
Before we begin, we note that all metric spaces are paracompact \cite{Stone},
so that, for any normed space $N$, its unit sphere $\sphere N$ is
paracompact with the metric induced from the norm.

We prove (1). By Corollary~\ref{cor:upsilon-and-delta-arelower-hemo-cont}(1)
and Michael's Selection Theorem (Theorem~\ref{thm:Michael's-Selection-Theorem}),
for every $\varepsilon>0$, there exists a continuous selection $\underline{\upsilon}:\sphere Z\to X$
of $\Upsilon_{\varepsilon}:\sphere Z\tocorr X$. We define $\upsilon$
by 
\[
\upsilon(\xi):=\begin{cases}
\norm{\xi}_{\infty}\underline{\upsilon}\parenth{\frac{\xi}{\norm{\xi}_{\infty}}} & \xi\neq0\\
0 & \xi=0
\end{cases}\quad(\xi\in Z),
\]
which is clearly positively homogeneous. That $\upsilon$ is continuous
follows from an application of the reverse triangle inequality while
keeping the boundedness of $\upsilon$ in mind.

We prove (2). By Corollary~\ref{cor:upsilon-and-delta-arelower-hemo-cont}(2)
and Michael's Selection Theorem (Theorem~\ref{thm:Michael's-Selection-Theorem}),
for every $\varepsilon>0$, there exists a continuous selection $\underline{\delta}:\sphere X\to\elloneomega X$
of $\Delta_{\varepsilon}:\sphere X\tocorr\elloneomega X$. By defining
\[
\delta(x):=\begin{cases}
\norm x\underline{\delta}\parenth{\frac{x}{\norm x}} & x\neq0\\
0 & x=0
\end{cases}\quad(x\in X),
\]
we are done, noting that positive homogeneity of $\delta$ is clear,
and that continuity follows again through application of the reverse
triangle inequality and the boundedness of $\delta$.
\end{proof}
The following two results, our Strong Klee-And\^o Theorems, now follow
from Corollary~\ref{cor:raw-klee-ando}. It is trivial that (3)$\Rightarrow$(2)$\Rightarrow$(1)
in both results below. The implication (1)$\Rightarrow$(3) in both
results below easily follow by fixing some $\varepsilon>0$, say $\varepsilon:=1$,
and applying Corollary~\ref{cor:raw-klee-ando}.
\begin{cor}[Strong Klee-And\^o Theorem for coadditivity]
\label{cor:klee-ando-for-coadd} Let $X$ be a Banach space and $\{C_{\omega}\}_{\omega\in\Omega}$
an indexed collection of closed cones in $X$. Let $Z$ be either
of the spaces $\comega X$ or $\ellinftyomega X$. Then the following
are equivalent:

\begin{enumerate}
\item For every $\xi\in Z$, the intersection $\bigcap_{\omega\in\Omega}(\xi_{\omega}+C_{\omega})$
is non-empty.\medskip
\item There exists an $\alpha>0$ such that, for every $\xi\in Z$, there
exists some $y\in\bigcap_{\omega\in\Omega}(\xi_{\omega}+C_{\omega})$
with $\norm y\leq\alpha\norm{\xi}_{\infty}$.\medskip
\item There exists an $\alpha>0$ and a continuous positively homogeneous
map $\upsilon:Z\to X$ such that, for every $\xi\in Z$, we have $\upsilon(\xi)\in\bigcap_{\omega\in\Omega}(\xi_{\omega}+C_{\omega})$
and $\norm{\upsilon(\xi)}\leq\alpha\norm{\xi}_{\infty}$.
\end{enumerate}
\end{cor}

{}
\begin{cor}[Strong Klee-And\^o Theorem for conormality]
\label{cor:klee-ando-for-conorm} Let $X$ be a Banach space and
$\curly{C_{\omega}}_{\omega\in\Omega}$ an indexed collection of closed
cones in $X$. The following are equivalent:

\begin{enumerate}
\item For every $x\in X$, there exists a decomposition $x=\sum_{\omega\in\Omega}c_{\omega}$,
with $c_{\omega}\in C_{\omega}$ for every $\omega\in\Omega$, and
satisfying $\sum_{\omega\in\Omega}\norm{c_{\omega}}<\infty$.\medskip
\item There exists an $\alpha>0$ such that, for every $x\in X$, there
exists a decomposition $x=\sum_{\omega\in\Omega}c_{\omega}$, with
$c_{\omega}\in C_{\omega}$ for every $\omega\in\Omega$, and satisfying
$\sum_{\omega\in\Omega}\norm{c_{\omega}}\leq\alpha\norm x$.\medskip
\item There exists an $\alpha>0$ and, for every $\omega\in\Omega$, there
exists a continuous positively homogeneous map $\delta_{\omega}:X\to C_{\omega}$
such that, for every $x\in X$, we have $x=\sum_{\omega\in\Omega}\delta_{\omega}(x)$
and $\sum_{\omega\in\Omega}\norm{\delta_{\omega}(x)}\leq\alpha\norm x$.
\end{enumerate}
\end{cor}

	\bibliographystyle{amsplain}
	\bibliography{bib}
\end{document}